\documentclass[10pt]{article}
\usepackage{mathrsfs}
\usepackage{amsfonts}
\usepackage{amsthm,amsmath,amssymb,anysize}
\usepackage[all]{xy}
\newtheorem{lemma}{Lemma}[section]
\newtheorem{theorem}[lemma]{Theorem}

\newtheorem{definition}[lemma]{Definition}

\usepackage[numbers,sort&compress]{natbib}
\marginsize{3.5cm}{3.5cm}{3.6cm}{3cm}
\numberwithin{equation}{section}

\title{\textsf{2-local derivations and biderivations of $\frak{sl}(2)$ on all simple modules} \footnote{ supported by the NSF of Heilongjiang
Province (YQ2020A005) and the NSF of China (12061029).}}
\author{\textsc{Shujuan Wang$^{1}$},
  \textsc{ Zhaoxin Li$^{2}$}\footnote{Correspondence:  LZX15765737480@163.com }
\;and \textsc{Xiaomin Tang$^{2}$}
  \\
 {\small \textit{$^1$Department of Mathematics, Shanghai Maritime University,}}\\
\small \textit{Shanghai 201306, China} \\
 \small\textit{$^2$School of Mathematical Sciences, Heilongjiang University,}\\ \small\textit{Harbin 150080, China} }
\date{ }

\begin{document}
\maketitle
\begin{quotation}
\small\noindent \textbf{Abstract}:
This paper generalizes
the concepts of 2-local derivations and biderivations (without  the skewsymmetric condition) of a finite-dimensional Lie algebra from the adjoint module to any finite-dimensional module,
and  determines all 2-local derivations and biderivations
of the 3-dimensional complex simple Lie algebra $\frak{sl}(2)$ on its any finite-dimensional simple module.

\vspace{0.2cm} \noindent{\textbf{Keywords}}: 2-local derivations,  biderivations, $\frak{sl}(2)$, simple modules

\vspace{0.2cm} \noindent{\textbf{Mathematics Subject Classification}}: 17B05, 17B10, 17B45
\end{quotation}

\setcounter{section}{0}
\section{ Introduction}

It has been proved that derivations and generalized derivations of algebras are  influential and far-reaching,
which  are very important subjects in the study of both
algebras and their generalizations.
The concepts of 2-local derivations and biderivations of Lie algebras (on the adjoint module)
are first introduced in 2015 and 2011 respectively (see \cite{Wang-Yu-Chen,2jubu4,2015}),
which have aroused the interest
of a great many authors in recent years, see \cite{11,tang,BenkovicZ,Du-Wang,Ghosseiri,Han-Wang-Xia,Wang-Yu,Xia-Wang-Han,Xu-Wang,Fan-Dai,Wang-Yu-Chen,2015,zhao,2jubu3,2jubu4,2jubu5}.
For any complex simple Lie algebra (on the adjoint module),
each 2-local derivation  must be a derivation (see \cite{2015});
and every (skewsymmetric) biderivation is an inner one (see \cite{Xia-Wang-Han,tang}).
In 2018, Bre$\check{s}$ar  and Zhao generalized the concept of skewsymmetric biderivations
of Lie algebras from the adjoint module to any module in the paper \cite{zhao}.
Motivated by Bre$\check{s}$ar  and Zhao's work on skewsymmetric  biderivations,
we aim to generalize
the notions of 2-local derivations and  biderivations  (without the skewsymmetric condition)
of a finite-dimensional Lie algebra from the adjoint module to any finite-dimensional module,
and determine all 2-local derivations and biderivations
of $\frak{sl}(2)$ on  any  simple module $V(n)$.
It is well-known that 3-dimensional complex simple Lie algebra $\frak{sl}(2)$
is so important that its simple modules determine the structure of semi-simple Lie algebras.
In this paper, we determine
all 2-local derivations and biderivations of $\frak{sl}(2)$ on $V(2)$,
which covers the result about $\frak{sl}(2)$ in the paper \cite{2015} and \cite{tang}.
Our main results are listed in the following,
in which $\mathrm{Bder}(\mathfrak{sl}(2),V(n))$ and $\mathrm{Ibder}(\mathfrak{sl}(2))$
are defined in Subsections \ref{1949} and \ref{1950}, respectively.
\begin{theorem}\label{1939}
Every 2-local derivation of $\mathfrak{sl}(2)$ on $V(n)$ is a derivation.
\end{theorem}
\begin{theorem}\label{202202211356}
$ \mathrm{Bder}(\mathfrak{sl}(2),V(n))= \begin{cases}
			\mathrm{Ibder}(\mathfrak{sl}(2))&n=2\\
			0& n\neq 2.
		\end{cases}$
\end{theorem}
\section{Preliminaries}
In this paper, let $L$ be a finite-dimensional Lie algebra over the complex number field $\mathbb{C}$
and  $V$ a finite-dimensional $L$-module.
Denote by $\mathrm {Hom}(L,V)$ the set consisting of linear maps from $L$ to $V$,
and  $V^L$  the maximal trivial submodule of $V$,
  that is,
  $$V^{L}=\{v\in V\mid xv=0, \forall x\in L\}.$$

\subsection{Definitions}\label{1949}
For $ D\in \mathrm {Hom}(L,V) $,
if
\begin{align*}
	D\left ( \left [x,y  \right ] \right )=xD(y)-yD(x), \quad \forall x,y\in L,
\end{align*}
then $ D $ is said to be a derivation of $L$ on $V$.
Denote by $\mathrm {Der}(L,V)$ the set consisting of all derivations of $L$ on $V$.
It is clear that $\mathrm {Der} (L,V)$ is a finite-dimensional vector space.
Let $D\in\mathrm{Der}(L,V)$.
If there exists $v\in V$, such that $D(x)=xv$ for any $x\in L$,
 then $D$ is said to be an inner determined by $v$, denoted by $D^{v}$.
Write
$$\mathrm{Ider}(L,V)=\left\{D^{v}\mid v\in V \right\}.$$
It is clear that $\mathrm{Ider}(L,V)$ is a subspace of $\mathrm {Der}(L,V)$.

In the following, we introduce the definitions
of 2-local derivations and biderivations of $L$ on $V$,
which generalizes the usual ones from the adjoint module to any finite-dimensional module.

\begin{definition}
Let $g: L\longrightarrow V$ be a map.
If for any $x, y\in L$, there exists $D_{x,y}\in \mathrm {Der} (L,V)$ (depending on $x, y$),
such that
$$g (x)=D_{x,y}(x),\quad g (y)=D_{x,y}(y),$$
then $g$ is said to be a 2-local derivation of $L$ on $V$.
\end{definition}
\begin{definition}\label{1617}
	Let $g: L \times  L\rightarrow V$ be a bilinear map.
If for any $x,~y,~z\in L$,
   $$g([x,y],z)=xg(y,z)-yg(x,z),$$
   $$g(x,[y,z])=yg(x,z)-zg(x,y),$$
then $g$ is said to be a biderivation of $L$ on $V$.
Denote by $\mathrm{Bder}(L,V)$ the set consisting of biderivations of $L$ on $V$.
\end{definition}

\subsection{Lemmas}
In this subsection, we give some technical lemmas which will be used in the future.
\begin{lemma}\label{base}
Let $V$ be a nontrivial and simple $L$-module. Then
$$
a_{1}D^{u_{1}}+a_{2}D^{u_{2}}+\cdots +a_{k}D^{u_{k}}=0
\Longleftrightarrow a_{1}u_{1}+a_{2}u_{2}+\cdots +a_{k}u_{k}=0,
 $$
where $u_{1},u_{2},\ldots,u_{k} \in V, a_{1},a_{2},\ldots,a_{k} \in  \mathbb{C}$.
In particular, 	
    \begin{center}
	$D^{u_{1}},\ldots ,D^{u_{k}}$ is a basis  of $\mathrm{Ider}(L,V)\Longleftrightarrow u_{1},\ldots,u_{k}$ is a basis  of $V$.
	\end{center}
\end{lemma}
\begin{proof}
	Let $ x\in L$. Then
	\begin{align*}
	&(a_{1}D^{u_{1}}+a_{2}D^{u_{2}}+\cdots +a_{k}D^{u_{k}})(x)\\
    &=a_{1}D^{u_{1}}(x)+a_{2}D^{u_{2}}(x)+\cdots +a_{k}D^{u_{k}}(x)\\
	&=a_{1}xu_{1}+a_{2}xu_{2}+\cdots +a_{k}xu_{k}\\
	&=x(a_{1}u_{1}+a_{2}u_{2}+\cdots +a_{k}u_{k})\\
	&=D^{a_{1}u_{1}+a_{2}u_{2}+\cdots +a_{k}u_{k}}(x).
    \end{align*}
That is,
$$
	a_{1}D^{u_{1}}+a_{2}D^{u_{2}}+\cdots +a_{k}D^{u_{k}}=D^{a_{1}u_{1}+a_{2}u_{2}+\cdots +a_{k}u_{k}},
 $$
which implies the ``if" direction.

Let $a_{1}D^{u_{1}}+a_{2}D^{u_{2}}+\cdots +a_{k}D^{u_{k}}=0$.
For any $x\in L$, we have
  $$
  (a_{1}D^{u_{1}}+a_{2}D^{u_{2}}+\cdots +a_{k}D^{u_{k}})(x)=0,
  $$
then
  $$
 D^{a_{1}u_{1}+a_{2}u_{2}+\cdots +a_{k}u_{k}}(x) =x(a_{1}u_{1}+a_{2}u_{2}+\cdots +a_{k}u_{k}).
  $$
It follows that
  $$
 a_{1}u_{1}+a_{2}u_{2}+\cdots +a_{k}u_{k} \in V^{L}.
  $$
Since $V$ is nontrivial and simple,  we get $V^{L}=0$.
 Then
  $$a_{1}u_{1}+a_{2}u_{2}+\cdots +a_{k}u_{k}=0.$$
\end{proof}

Denote by $\mathrm{Blin}(L, V)$ the set consisting of bilinear maps from $L\times L$ to $V$.
Let $\frak{h}$ be a Cartan subalgebra of $L$.
Suppose
$L$ and  $V$ possess  weight space decompositions with respect to $\frak{h}$:
$$L=\frak{h} \bigoplus \sum_{\alpha \in \frak{h}^* }L_\alpha , \qquad V=\bigoplus_{\alpha \in\frak{h}^*}V_{\alpha}.$$
In this manner, $L$ is also an $\frak{h}^*$-graded Lie algebra and $V$ an $\frak{h}^*$-graded module.
Then the assumption on the dimensions implies that the $L$-modules $\mathrm{Hom}(L,V)$ and $\mathrm{Blin}(L,V)$
have $\frak{h}^*$-grading structure induced by the ones of $L$ and $V$:
\begin{align*}
\mathrm{Hom}(L,V)&=\bigoplus _{\gamma \in \frak{h}^*}\mathrm{Hom}(L,V)_{(\gamma)},\\
\mathrm{Blin}(L,V)&=\bigoplus _{\gamma \in \frak{h}^*}\mathrm{Blin}(L,V)_{(\gamma)},
\end{align*}
where
\begin{align*}
\mathrm{Hom}(L,V)_{(\gamma)}&=\{\phi\in \mathrm{Hom}(L,V)\mid \phi(L_{\alpha})\subset V_{\alpha+\gamma}, \forall \alpha \in  \frak{h}^*\},\\
\mathrm{Blin}(L,V)_{(\gamma)}&=\{\phi\in \mathrm{Blin}(L,V)\mid \phi(L_{\alpha},L_{\beta})\subset V_{\alpha+\beta+\gamma}, \forall \alpha,\beta \in  \frak{h}^*\}.
\end{align*}
Moreover, $\mathrm{Hom}(L,V)$  and $\mathrm{Blin}(L,V)$ are $L$-modules with the following action:
$$(xf)(y)=xf(y)-f([x,y]),$$
$$(xg)(y,z)=xg(y,z)-g([x,y],z)-g(y,[x,z]),$$
where $x,y,z\in L, f\in \mathrm{Hom}(L,V), g\in \mathrm{Blin}(L,V).$
Clearly,
$\mathrm{Hom}(L,V)_{(\gamma)}$ and $\mathrm{Blin}(L,V)_{(\gamma)}$ are precisely
the weight spaces of weight $\gamma$ with respect to $\frak{h}$:
\begin{align*}
\mathrm{Hom}(L,V)_{(\gamma)}&=\mathrm{Hom}(L,V)_{\gamma}\\
&=\{\varphi\in \mathrm{Hom}(L,V)\mid h \varphi=\gamma(h)\varphi, \forall h\in \frak{h}\},
\end{align*}
\begin{align*}
\mathrm{Blin}(L,V)_{(\gamma)}&=\mathrm{Blin}(L,V)_{\gamma}\\
&=\{\varphi\in \mathrm{Blin}(L,V)\mid h \varphi=\gamma(h)\varphi, \forall h\in \frak{h}\}.
\end{align*}
It is a standard fact that $\mathrm{Der}(L,V)$ is an $\frak{h}^*$-graded submodule
(hence a weight-submodule with respect to $\frak{h}$) of $\mathrm{Hom}(L,V)$:
\begin{align}\label{1543}
\mathrm{Der}(L,V)&=\bigoplus _{\gamma \in \frak{h}^*}\mathrm{Der}(L,V)\cap \mathrm{Hom}(L,V)_{\gamma}.
\end{align}
\begin{lemma}\label{1311}
Keep the notations as above. Then
\begin{align*}
\mathrm{Bder}(L,V)&=\bigoplus_{\gamma \in\frak{h}^*}\mathrm{Bder}(L,V)\cap \mathrm{Blin}(L,V)_{\gamma}.
\end{align*}
\end{lemma}
\begin{proof}
For any $g\in \mathrm{Bder}(L,V)$, there exists $g_{\gamma}\in  \mathrm{Blin}(L,V)_{\gamma}$
such that $g=\sum_{\gamma \in\frak{h}^*}g_{\gamma}$.
If $x\in L_{\beta}$ with $\beta\in \frak{h}^*$,
then $g(x,\cdot)=\sum_{\gamma \in\frak{h}^*}g_{\gamma}(x,\cdot)$.
Note that $g(x,\cdot)\in \mathrm{Der}(L,V)$ and $g_{\gamma}(x,\cdot)\in \mathrm{Hom}(L,V)_{\gamma+\beta}$.
Then by Equation (\ref{1543}), we have
$g_{\gamma}(x,\cdot)\in \mathrm{Der}(L,V)_{\gamma+\beta}$.
It follows that
$g_{\gamma}(y,\cdot)\in \mathrm{Der}(L,V)$ for any $y\in L$.
Similarly, $g_{\gamma}(\cdot,y)\in \mathrm{Der}(L,V)$ for any $y\in L$.
Then $g_{\gamma}\in \mathrm{Bder}(L,V)$.
\end{proof}
\begin{lemma}\label{937}
	Let $V$ be a nontrivial and simple $L$-module.
If $\mathrm{Der}(L,V)=\mathrm{Ider}(L,V)$,
then for any $g \in \mathrm{Bder} (L,V)_\alpha$ with $\alpha \in\frak{h}^*$,
 there exist two maps $\varphi ,\psi \in \mathrm{Hom}(L,V)_\alpha $ such that
	$$g (x,y)=x\varphi(y)=y\psi (x),\qquad  \forall x,y\in L .$$
\end{lemma}
\begin{proof}
Let $x\in L$.
By definition we have $g(x,\cdot)\in \mathrm{Der}(L,V)$ and $g(\cdot,x)\in \mathrm{Der}(L,V)$.
Since $\mathrm{Der}(L,V)=\mathrm{Ider}(L,V)$ and $V$ is a nontrivial and simple $L$-module,
then there exists the unique $v\in V$ such that $g(x,\cdot )=D^v,$
and the unique $w\in V$ such that $g(\cdot, x)=D^w,$
that is,
	$$g(x,y)=D^v(y)=yv, \; g(y,x)=D^w(y)=yw, \;\;  \forall y\in L.$$
Define the maps $\psi: L\rightarrow V$ by $\psi(x)=v$
and $\varphi: L\rightarrow V$ by $\varphi(x)=w$,
which means
$$g(x,y)=y\psi(x), \; g(y,x)=y\varphi(x), \;\;  \forall y\in L.$$
Claim $\psi \in \mathrm{Hom}(L,V)$.
In fact, for any $a\in\mathbb{C}$ and $x_1,x_2,y\in L,$
	the equation $$g(ax_1+x_2,y)=ag(x_1,y)+g(x_2,y)$$
implies
	$$y\psi (ax_1+x_2)=ay\psi (x_1)+y\psi (x_2)=y(a\psi (x_1)+\psi (x_2)).$$
Since $y$ is arbitrary, we have
$$\psi (ax_1+x_2)-a\psi (x_1)-\psi (x_2)\in V^L=0,$$
where the last equality holds since $V$ is nontrivial and simple.
Hence
$$\psi (ax_1+x_2)=a\psi (x_1)+\psi (x_2),$$
that is, $\psi \in \mathrm{Hom}(L,V)$.
In the following, we shall prove $\psi \in \mathrm{Hom}(L,V)_{\alpha}$.
Let $\beta ,~\gamma \in \frak{h}^*,$ $ x\in L_{\beta}, y\in L_{\gamma}$ and $h\in \frak{h}$.
On one hand,
	$$hg(x,y)=(\alpha +\beta +\gamma )(h)g(x,y)=(\alpha +\beta +\gamma )(h)y\psi (x).$$
On the other hand,
	\begin{align*}
		hg(x,y)=hy\psi (x)&=[h,y]\psi (x)+yh\psi (x)\\
	                        	&=\gamma (h)y\psi (x)+yh\psi (x)\\
	                          	&=y(\gamma (h)\psi (x)+h\psi (x)).
	\end{align*}
Then
$$(\alpha +\beta )(h)y\psi (x)=yh\psi (x).$$
Since both $\gamma$ and $y$ are arbitrary, so
$$h\psi (x)-(\alpha +\beta )(h)\psi (x)\in V^L=0,$$
that is,
$$h\psi (x)=(\alpha +\beta )(h)\psi (x).$$
It follows that $\psi \in \mathrm{Hom}(L,V)_\alpha $.
Similarly, $\varphi \in \mathrm{Hom}(L,V)_\alpha $.
\end{proof}

Recall that 3-dimensional simple Lie algebra $\mathfrak{sl}(2)$
has a standard basis:
$$
	e:=\begin{pmatrix}
		0& 1\\
		0& 0
	\end{pmatrix},\quad
	h:=\begin{pmatrix}
		1& 0\\
		0& -1
	\end{pmatrix},\quad
	f:=\begin{pmatrix}
		0& 0\\
		1& 0
	\end{pmatrix}.
$$
As we are all known, $n+1$-dimensional simple $\mathfrak{sl}(2)$-module $V(n)$ has a standard basis:
$$v _{0},v _{1},\ldots,v _{n},$$
with the action as follows:
$$
ev _{i}=i\left ( n+1-i \right )v_{i-1},~ \quad hv _{i}=(n-2i)v_{i},~ \quad fv _{i}=v_{i+1},
$$
where $i=0,1,\ldots ,n,~ v_{-1}=v_{n+1} =0$.
Note that
$\mathbb{C}h$ is a Cartan subalgebra of $\mathfrak{sl}(2)$,
and
$$\mathfrak{sl}(2)=\mathbb{C}h \bigoplus \mathfrak{sl}(2)_2\bigoplus \mathfrak{sl}(2)_{-2},\quad V(n)=\bigoplus _{i=0}^{n}V(n)_{n-2i},$$
where
$$\mathfrak{sl}(2)_2=\mathbb{C}e,\quad \mathfrak{sl}(2)_{-2}=\mathbb{C}f,\quad V(n)_{n-2i}=\mathbb{C}v_i.$$
Write
$$\Delta _V=\left\{n-2i\mid i=0,1,\cdots ,n \right\}, \quad \Delta =\left\{0,2,-2\right\}.$$

The following lemma characterizes the weight spaces
$\mathrm{Bder} (\mathfrak{sl}(2),V(n))_{\gamma}$ for
$\gamma \notin \Delta _V.$
\begin{lemma}\label{1028}
	Let $n\geq 1$. If $\gamma \notin \Delta _V,$
then $\mathrm{Bder} (\mathfrak{sl}(2),V(n))_{\gamma}=0$.
\end{lemma}
\begin{proof}
Claim that $\mathrm{Bder}(\mathfrak{sl}(2),V(n))_{\gamma}=0$
if $\gamma \notin \Delta _V\cup \Delta$.
In fact,
Let $g \in \mathrm{Bder}(\mathfrak{sl}(2),V(n))_{\gamma}$.
Then by Lemma \ref{937}, there exist $\varphi ,\psi \in \mathrm{Hom}(\mathfrak{sl}(2),V(n))_{\gamma}$  such that
	$$g(x,y)=x\varphi (y)=y\psi  (x),\quad \forall x,y\in \mathfrak{sl}(2).$$
	Let $x=h$ and $y\in \mathfrak{sl}(2)_\eta$ with $\eta \in\Delta $.
Then
$$\varphi (y)\in V(n)_{\gamma +\eta},\quad \psi (h)\in V(n)_\gamma =0.$$
It follows that
	$$g(h,y)=h\varphi (y)=y\psi (h)=0,$$
	and then
	$\varphi (y)\in V(n)_0$.
Hence we have
	$\varphi (y)\in V(n)_{\gamma +\eta }\cap   V(n)_0.$
Note that $\gamma +\eta \neq0$, and then $\varphi (y)=0$.
Since $\eta$ and $y$ are arbitrary, so $\varphi=0$.

Below we shall prove that $\mathrm{Hom}(\mathfrak{sl}(2),V(n))_\gamma =0$
if $n$ is odd and $\gamma \in\Delta$.
Let	$\varphi \in \mathrm{Hom}(\mathfrak{sl}(2),V(n))_\gamma$.
Then
$$\varphi (e)\in V(n)_{2+\gamma },\quad
	\varphi (f)\in V(n)_{-2+\gamma },\quad
\varphi (h)\in V(n)_\gamma.$$
Since $n$ is odd and $\gamma \in\Delta$,
so $V(n)_{\pm 2+\gamma }=V(n)_\gamma =0$,
which implies $\varphi=0$.
\end{proof}

\section{Proof of the Main Results}

In this section, we shall determine all 2-local derivations
and biderivations of $\mathfrak{sl}(2)$ on $V(n)$.
\subsection{2-local derivations}
\textbf{Proof of Theorem \ref{1939}}
Let $g$ be a 2-local derivation of $\mathfrak{sl}(2)$ on $V(n)$.
In the following we shall prove that $g$ is a derivation by steps.

	\textbf{Step 1} By definition, there exists $ D_{h,e}\in \mathrm {Der}(\mathfrak{sl}(2),V(n))$,
such that
$$g(h)=D_{h,e}(h), \qquad g (e)=D_{h,e}(e),$$
that is,
	$$
		(g-D_{h,e})(h)=g (h)-D_{h,e}(h)=0,\quad
		 (g-D_{h,e})(e)=g (e)-D_{h,e}(e)=0.$$
   Note that $g_1:=g-D_{h,e}$ is also a 2-local derivation
   of $\mathfrak{sl}(2)$ on $V(n)$ and
   $g_1(h)= g_1(e)=0.$

   \textbf{Step 2}
By definition, for all $x,y\in \mathfrak{sl}(2)$,
there exists $D_{x,y}\in \mathrm {Der}(\mathfrak{sl}(2),V(n))$,
such that
$$g_1 (x)=D_{x,y}(x), \quad g_1 (y)=D_{x,y}(y).$$
By Lemma \ref{base},
$\mathrm{Ider}(\mathfrak{sl}(2),V(n))$ has a basis:
$D^{v_{0}},D^{v_{1}},\ldots,D^{v_{n}}.$
Since the first cohomology of $\mathfrak{sl}(2)$ with the coefficients  $V(n)$ is trivial (see \cite{1318}),
we have
$$\mathrm{Ider}(\mathfrak{sl}(2),V(n))=\mathrm{Der}(\mathfrak{sl}(2),V(n)).$$
That is, $\mathrm {Der}(\mathfrak{sl}(2),V(n))$
has a basis: $ D^{v_{0}},D^{v_{1}},\ldots ,D^{v_{n}}.$
Thus we may assume that
$$D_{x,y}=\sum_{i=0}^{n}a_{i}(x,y)D^{v_{i}}=D^{v_{x,y}},\quad a_{i}(x,y)\in \mathbb{F},$$
where $v_{x,y}=\sum_{i=0}^{n}a_{i}(x,y)v_{i}.$
It follows that
\begin{align*}
	 g_1 (x)&=D_{x,y}(x)=x\left(\sum_{i=0}^{n}a_{i}(x,y)v_{i}\right), \\
	 g_1 (y)&=D_{x,y}(y)=y\left(\sum_{i=0}^{n}a_{i}(x,y)v_{i}\right).
\end{align*}

\textbf{Step 3} On one hand, by Step 2 we get
\begin{align*}
	 g_1 (f)&=f\left(\sum_{i=0}^{n}a_{i}(f,e)v_{i}\right)=\sum_{i=0}^{n}a_{i}(f,e)v_{i+1}, \\
	 g_1 (e)&=e\left(\sum_{i=0}^{n}a_{i}(f,e)v_{i}\right)=\sum_{i=0}^{n}a_{i}(f,e)i(n+1-i)v_{i-1}.
\end{align*}
Since $g _{1}(e)=0 $ by Step 1, we get
  $i(n+1-i)a_{i}(f,e)=0$ for $1\leq i\leq n,$
 which implies that $ a_{i}(f,e)=0$ for $1\leq i\leq n.$
 Then
   \begin{equation}\label{2122}
   	 	g _{1}(f)= a_{0}(f,e)v_{1}.
   \end{equation}
On the other hand, by Step 2 we get
\begin{align*}
	 g_1 (f)&=f\left(\sum_{i=0}^{n}a_{i}(f,h)v_{i}\right)=\sum_{i=0}^{n}a_{i}(f,h)v_{i+1}, \\
	 g_1 (h)&=h\left(\sum_{i=0}^{n}a_{i}(f,h)v_{i}\right)=\sum_{i=0}^{n}a_{i}(f,h)(n-2i)v_{i}.
\end{align*}
Since $g_{1}(h)=0 $ by Step 1, we get
$a_{i}(f,e)(n-2i)=0$ for $0\leq i\leq n.$

If $n$ is odd, then $a_i(f,h)=0$, which implies $g_1(f)=0$.

If $n$ is even and nonzero, then $a_i(f,h)=0$ for $i\neq \frac{n}{2}$,
which implies
\begin{equation}\label{2123}
   g_{1}(f)=a_{\frac{n}{2}}(f,h)v_{\frac{n}{2}+1}.
   \end{equation}
   By Equations (\ref{2122}) and (\ref{2123}), we get
  $$
   	a_{0}(f,e)v_{1}=a_{\frac{n}{2}}(f,h)v_{\frac{n}{2}+1},\qquad \mbox{if $n$ is even and nonzero}.
  $$
When $n$ is even and nonzero,
$v_1, v_{\frac{n}{2}+1}$ are linearly independent, then $a_{0}(f,e)=a_{\frac{n}{2}}(f,h)=0 $,
which means $g_{1}(f)=0$.

If $n=0$, then $V(n)$ is a trivial $\frak{sl}(2)$-module.
It is clear that $g_{1}(f)=0$ by the definition of 2-local derivations.

In summary,  $g_{1}(f)=0$.

\textbf{Step 4}	
Let $x=x_{e}e+x_{h}h+x_{f}f \in \mathfrak{sl}(2)$,
where $x_e,x_h,x_f\in \mathbb{C}$.
By Steps 1 and 2, we get
\begin{align*}
	g_{1}(f)&=f\left(\sum_{i=0}^{n}a_{i}(x,f)v_{i}\right)=\sum_{i=0}^{n}a_{i}(x,f)v_{i+1}=0,\\
     g_{1}(e)&=e\left(\sum_{i=0}^{n}a_{i}(x,e)v_{i}\right)=\sum_{i=0}^{n}a_{i}(x,e)i(n+1-i)v_{i-1}=0.		
\end{align*}
Then $a_i(x,f)=0$ for $0 \leq i \leq n-1$ and
$a_j(x,e)=0$ for $1 \leq j \leq n$.
It follows that
\begin{align}
	g_{1}(x)&=xa_{n}(x,f)v_{n}=(x_{e}e+x_{h}h)a_{n}(x,f)v_n=a_{n}(x,f)(nx_ev_{n-1}-nx_{h}v_{n}),\label{1255}\\
	g_{1}(x)&=xa_{0}(x,e)v_{0}=(x_{h}h+x_{f}f)a_0(x,e)v_0=a_{0}(x,e)(nx_hv_0+x_fv_1).\label{1254}
\end{align}
By Equations (\ref{1255}) and (\ref{1254}), we get
$$a_{0}(x,e)(nx_hv_0+x_fv_1)=a_{n}(x,f)(nx_ev_{n-1}-nx_{h}v_{n}).$$
If $n>2$, then  $a_{0}(x,e)=0$ or $ a_{n}(x,f)=0 $, and $ g_{1}(x)=0. $

By Steps 1 and 2, we get
\begin{align*}
    g_{1}(x)&=x\left(\sum_{i=0}^{n}a_{i}(x,h)v_{i}\right)=\sum_{i=0}^{n}a_{i}(x,h)xv_i,\\
	g_{1}(h)&=h\left(\sum_{i=0}^{n}a_{i}(x,h)v_{i}\right)=\sum_{i=0}^{n}(n-2i)a_{i}(x,h)v_{i}=0.	
\end{align*}
If $n=1$, then all $a_{i}(x,h)=0$, which  implies $g_{1}(x)=0$.

Let $n=2$ in the following.
Then  $a_{i}(x,h)=0$ for $i \neq  1$.
It follows that
\begin{equation}\label{1256}
	\begin{aligned}
	g_{1}(x)&=xa_{1}(x,h)v_{1}=(x_{e}e+x_{f}f)a_{1}(x,h)v_{1} \\
	&=a_{1}(x,h)\left(2x_{e}v_{0}+x_{f}v_{2}\right).
	\end{aligned}
\end{equation}
By Equations (\ref{1256}) and (\ref{1254}), we get
$$a_{0}(x,e)(2x_hv_0+x_fv_1)=a_{1}(x,h)(2x_{e}v_{0}+x_{f}v_{2}).$$
If $x_f\neq 0$, then $a_{1}(x,h)=0$ and $a_{0}(x,e)=0$, which means $g_{1}(x)=0$.
Then by Steps (1) and (3), it is sufficient to assume $x_{f}=0, ~x_{e} x_{h}\neq 0$ in the following.
Considering Equations (\ref{1256}) and (\ref{1255}), we get
	$$	a_{2}(x,f)(2x_ev_{1}-2x_{h}v_{2})=a_{1}(x,h)(2x_{e}v_{0}+x_{f}v_{2}),$$
which means $ a_{1}(x,h)= a_{2}(x,f)=0.$
Then by Equation (\ref{1256}), we get $g_{1}(x)=0.$

In summary,  $g_{1}(x)=0$ for any $x\in \frak{sl}(2)$, that is $g_{1}=0$.

\textbf{Step 5}	
By Steps 1 and 4, we get $g=D_{h,e}$.

\subsection{Biderivations}\label{1950}
As we are known,
$V(2)$ is isomorphic to the adjoint module of $\frak{sl}(2)$
by setting $v_{0}=e$.
Then for $\lambda\in \mathbb{C}$, it is easy to verify that the  map
$$f: \mathfrak{sl}(2)\times \mathfrak{sl}(2)\longrightarrow V(2), \quad (x,y)\longmapsto\lambda [x,y]$$
is a biderivation of $\mathfrak{sl}(2)$ on $V(2)$, which is said to be inner.
Write $\mathrm{Ibder}(\mathfrak{sl}(2))$ for all inner biderivations
from  $\mathfrak{sl}(2)$ to $V(2)$.

\textbf{Proof of Theorem \ref{202202211356}}
Let $g\in \mathrm{Bder}(\mathfrak{sl}(2),V(0))$.
Note that $V(0)$ is a trivial $\mathfrak{sl}(2)$-module,
then by Definition \ref{1617},
we have
$$g([x,y],z)=g(x,[y,z])=0\qquad \forall x,y,z\in\mathfrak{sl}(2).$$
Since $[\mathfrak{sl}(2), \mathfrak{sl}(2)]=\mathfrak{sl}(2)$,
so $g=0.$

In the following, we assume that $n\neq 0$.
Let $g\in \mathrm{Bder}(\mathfrak{sl}(2),V(n))$.
Define $g_1,g_2\in \mathrm{Blin}(\mathfrak{sl}(2),V(n))$ by
$$g_1(x,y)=(g(x,y)+g(y,x))/2, \quad g_2(x,y)=(g(x,y)-g(y,x))/2,$$
where $x,y\in L$.
Then every $g_i$ is a biderivation and
$$g=g_1+g_2,\quad g_1(x,y)=g_1(y,x),\quad g_2(x,y)=-g_2(y,x),$$
where $x,y\in L$.
Schur's Lemma tells us that
every $\frak{sl}(2)$-module homomorphism
from $\frak{sl}(2)$ to $V(n)$
must be of the form $\delta_{n,2}\lambda\mathrm{id}_{\frak{sl}(2)}$,
where $\lambda\in \mathbb{C}$ and $\mathrm{id}_{\frak{sl}(2)}$
is the identity map on $\frak{sl}(2)$.
Hereafter the symbol $\delta_{i,j}$ means 1 if $i=j$, and 0 otherwise.
Then by \cite[Theorem 2.3]{zhao},
we have that $g_2(x,y)=\delta_{n,2}\lambda[x,y]$
for $\lambda\in \mathbb{C}$ and $x,y\in L$.

Below it is sufficient to show $g_1=0$.
By Lemmas \ref{1311} and  \ref{1028},
we may assume $g_1=g_{10}+g_{11}+\cdots+g_{1n}$,
where
$g_{1i}\in \mathrm{Bder}(\mathfrak{sl}(2),V(n))_{n-2i}$ for $i=0,1,\ldots,n$.
By Lemma \ref{937} there exist two maps $\varphi_{i} ,\psi_{i} \in \mathrm{Hom}(\mathfrak{sl}(2),V(n))_{n-2i},$
such that
$$
	x\psi_{i}(y)=g_{1i} (y,x)=g_{1i} (x,y)=x\varphi_{i} (y), \quad \forall x,y\in \mathfrak{sl}(2).
$$
It follows that $\psi_i=\varphi_i,$ and
    $$\varphi_{i}(h)\in V(n)_{n-2i},\quad \varphi_{i}(e)\in V(n)_{n-2(i-1)},\quad\varphi_{i}(f)\in V(n)_{n-2(i+1)},$$
    \begin{align}\label{1024}
g_{1i} (x,y)&=x\varphi_{i} (y)=y\varphi_{i} (x), \quad \forall x,y\in \mathfrak{sl}(2).
\end{align}
Then we assume that
$$\varphi_{i}(h)=a_{hi}v_{i},\quad \varphi_{i}(e)=a_{ei}v_{i-1},\quad \varphi_{i}(f)=a_{fi}v_{i+1},$$
where $a_{*},\in \mathbb{C}.$
Then by Equation (\ref{1024}), we get
$$h\varphi_{i}(e)=e\varphi_{i}(h),\quad h\varphi_{i}(f)=f\varphi_{i}(h),\quad e\varphi_{i}(f)=f\varphi_{i}(e),$$
that is,
$$(n-2i+2)a_{ei}v_{i-1}=i(n+1-i)a_{hi}v_{i-1},$$
$$(n-2i-2)a_{fi}v_{i+1}=a_{hi}v_{i+1},$$
$$(i+1)(n-i)a_{fi}v_{i}=(1-\delta_{i, 0})a_{ei}v_{i}.$$
It follows that
\begin{align}
(1-\delta_{i, n})a_{hi}&=(1-\delta_{i, n})(n-2i-2)a_{fi}\label{5-2},\\
(1-\delta_{i, 0})(n-2i+2)a_{ei}&=(1-\delta_{i, 0})i(n+1-i)a_{hi}\label{5-3},\\
 (1-\delta _{i, 0})a_{ei}&=(i+1)(n-i)a_{fi}\label{5-4}.
\end{align}
In the following, we shall determine $\varphi _{i}$ and $\psi _{i}$ for $0\leq i\leq n$ by steps.
\begin{itemize}
\item[(a)]
If $i=0 $, then
$$
a_{fi}\stackrel{(\ref{5-4})}{=}0,\quad
a_{hi}\stackrel{(\ref{5-2})}{=}(n-2)a_{fi}.
$$
It follows that $g _{10}=0.$
\item[(b)]
If $i=n$, then
$$a_{ei}\stackrel{(\ref{5-4})}{=}0,\quad (2-n)a_{ei}\stackrel{(\ref{5-3})}{=}na_{hi}.$$
It follows that $g _{1n}=0.$
\item[(c)]
If $i\neq 0$ and $i\neq n$, then by Equations (\ref{5-2})-(\ref{5-4}),
we have
$$(i+1)(n-i)(n-2i+2)a_{fi}=i(n+1-i)(n-2i-2)a_{fi},$$
which implies $a_{fi}=0$.
It follows that $a_{ei}\stackrel{(\ref{5-4})}{=}0\stackrel{(\ref{5-2})}{=}a_{hi}$.
Hence $g_{1i} =0.$
\end{itemize}	
From (a)-(c), it follows that $g_1$=0.

\end{document}